\documentclass[12pt,reqno]{amsart}
\usepackage[english]{babel}
\usepackage{times}

\newcommand{\cc}{\mathbb{C}}

\newcommand{\pn}{{\mathbb{P}}^{n-1}}

\newcommand{{\rr}}{\mathcal{R}}

\newcommand{\sfc}{\sigma_{f}^{c}}

\newtheorem{thm}{Theorem}[section]
\newtheorem{defn}[thm]{Definition}
\newtheorem{lem}[thm]{Lemma}
\newtheorem{corr}[thm]{Corollary}
\newtheorem{prop}[thm]{Proposition}
\newtheorem{exam}[thm]{Example}
\begin{document}

 \title{Projective spectrum and cyclic cohomology}

\author[P. Cade]{Patrick Cade}
\address{Department of Mathematics, Siena College, Loudonville, NY 12211, U.S.A.}
\email{pcade@siena.edu}

\author[R. Yang]{Rongwei Yang }
\address{Department of Mathematics and Statistics, SUNY at Albany, Albany, NY 12047, U.S.A.}
\email{ryang@albany.edu}



\begin{abstract}
For a tuple $A=(A_1,\ A_2,\ ...,\ A_n)$ of elements in a unital algebra ${\mathcal B}$ over $\cc$, its {\em projective spectrum} $P(A)$ or $p(A)$
is the collection of $z\in \cc^n$, or respectively $z\in \pn$ such that the multi-parameter pencil $A(z)=z_1A_1+z_2A_2+\cdots +z_nA_n$ is not invertible in ${\mathcal B}$. 
${\mathcal B}$-valued $1$-form $A^{-1}(z)dA(z)$ contains much topological information about $P^c(A):=\cc^n\setminus P(A)$. In commutative cases, invariant multi-linear functionals are effective tools to extract that information. This paper shows that in non-commutative cases, the cyclic cohomology of ${\mathcal B}$ does a similar job. In fact, a Chen-Weil type map $\kappa$ from the cyclic cohomology of ${\mathcal B}$ to the de Rham cohomology $H^*_d(P^c(A),\ \cc)$ is established. As an example, we prove a closed high-order form of the classical Jacobi's formula.
\end{abstract}
\maketitle

\section{Introduction}

Multivariable operator theory at its current stage has focused on study of commuting operator tuples. A center piece of this study is the Taylor spectrum defined in [Ta].
Study of noncommuting tuples, on the other hand, were carried out mostly by scholars in the field of Differential Equations and Mathematical Physics. In recent years, however, functional analysts are getting more and more interested in noncommuting tuples, and many feel a great need to bring noncommuting tuples into the domain of multivariable operator theory. In fact, very promising work has been done by several researchers, for instance, Ball and Bolotnikov [BB], Greene, Helton and Vinnikov [GHV],
Helton, Klep and McCullough [HKM], and a great amount of work by Popescu (for example [Po]).
Whether there is a good definition of joint spectrum for noncommuting tuples now becomes an important question. 

For a tuple $A=(A_1,\ A_2,\ ...,\ A_n)$ of elements in a unital algebra ${\mathcal B}$ over $\cc$, the linear sum $A(z)=z_1A_1+z_2A_2+\cdots +z_nA_n$ is called a multiparameter pencil for $A$. First studies of $A(z)$ when $A$ is a tuple of linear operators were made by H. Weyl and physicist R. Feynman in search of a functional calculus for noncommuting operator tuples. This line of research continues to this date (for example, Anderson and S\"{o}strand [AS], F. Colombo, G. Gentili, I. Sabadini and D. Struppa [CGSS], Jefferies [Je]).
Invertibility of $A(z)$ is studied in partial differential equations ( Atkinson [At], Sleeman [Sl]), and in simulation of electronic circuits and in fluid mechanics (Hochstenbach [Ho]). Some pure mathematical studies were carried out as well, for example see Vinnikov ([Vi]). For practical reasons, attention in the past was focused on the case when $A$ is a tuple of self-adjoint operators. But theorists are not confined by practical needs. Inspired by these works, the second author proposed the so-called projective spectrum in [Ya]. To be precise, the {\em projective spectrum} $P(A)$ or $p(A)$ for a general tuple $A$ is the collection of $z\in \cc^n$, or respectively $z\in \pn$ such that $A(z)=z_1A_1+z_2A_2+\cdots +z_nA_n$ is not invertible in ${\mathcal B}$. The {\em projective resolvent sets} refer to
their complements $P^c(A)={\cc}^n\setminus P(A)$ and $p^c(A)=\pn \setminus p(A)$. It is shown in [Ya] that for every tuple $A$, $p(A)$ is a nontrivial compact subset of $\pn$. Quite a few examples were given there. Here we just mention two of them.

\begin{exam}
 Let $A_1$ be any element in ${\mathcal B}$, and $A_2=-I$. Then for the tuple $A=(A_1,\ A_2)$,
$A(z)=z_1A_1-z_2I$. Clearly, if $[z_1,\ z_2]$ is in $p(A)$
then $z_1\neq 0$, and $p(A)$, under the affine coordinate $z_2/z_1$, is the classical spectrum $\sigma(A_1)$. 
\end{exam}

\begin{exam}
When ${\mathcal B}$ is the matrix algebra $M_k(\cc)$, $A=(A_1,\ A_2,\ ...,\ A_n)$ is a tuple of $k\times k$ matrices.
Then $A(z)$ is invertible if and only if $detA(z)\neq 0$. Since $detA(z)$ is homogeneous of degree $k$,
\[p(A)=\{z=[z_1,\ ,z_2,\ ...,\ z_n]\in \pn:\ detA(z)=0\},\]
which is a projective hypersurface of degree $k$. $p^c(A)$ in this case is a hypersurface complement.
\end{exam}

Of course, one shall not expect $p(A)$ to be a hypersurface if ${\mathcal B}$ is an infinite dimensional algebra. Nevertheless, the resolvent set $P^c(A)$
is still manageable. For instance, it is shown in [Ya] that if ${\mathcal B}$ is a $C^*$-algebra then every path-connected component of $P^c(A)$ is a domain of holomorphy.
The Maurer-Cartan type ${\mathcal B}$-valued $1$-form $\omega_A(z):=A^{-1}(z)dA(z)$ contains much topological information about $P^c(A)$. Here, as usual, $d=\partial +\overline{\partial}$. Since this paper only concerns with holomorphic functions, $\overline{\partial}$ does not play any role here. In commutative cases, invariant multi-linear functionals are effective tools to extract that information. This paper shows that in non-commutative cases, the cyclic cohomology of ${\mathcal B}$ does a similar job. In fact, a Chen-Weil type map $\kappa$ from the cyclic cohomology of ${\mathcal B}$ to the de Rham cohomology $H^*_d(P^c(A),\ \cc)$ is established. The classical Jacobi's formula states that for a differentiable square-matrix-valued function $f(z)$,
\[tr(f^{-1}(z)df(z))=d{\text{log}} detf(z)\]
on the set of $z$ where $f(z)$ is invertible. As an example, we prove a closed high-order form of this formula. \\

{\bf Acknowledgment}: A large part of this paper comes from the first author's dissertation at SUNY at Albany. The authors thank Professor M. Varisco for many helpful discussions on the topic of cyclic cohomology. 

\section{Preliminaries}

This section recalls a few notions and results from [Ya]. We will review what has been done in the commutative case.
It serves as a prelude to what we will develop for the non-commutative case.
Throughout the paper, for a domain $D\subset \cc^n$, we let 
$H^{*}_d(D,\cc)$ denote the de Rham cohomology of $D$.
It should be noted that if $D$ is a domain of holomorphy, then $H^p_d(D,\ \cc)$ is spanned by holomorphic forms for each $p$ (cf. [Ra]).  

First, one observes that for a linear functional $\phi$ on ${\mathcal B}$, $\phi(\omega_A(z))=\sum_{j=1}^n\phi(A^{-1}(z)A_j)dz_j$ 
is a holomorphic $1$-form on $P^c(A)$. Likewise, for a $k$-linear functional $F$,
$F(\omega_A(z),\ \omega_A(z),\ ...,\ \omega_A(z))$ is a holomorphic $k$-form on $P^c(A)$.

A $k$-linear functional $F$ on ${\mathcal B}$ is said to be invariant if 
\begin{equation*}
F(a_1,\ a_2,\ ,\ ...,\ a_k)=F(ga_1g^{-1},\ ga_2g^{-1},\ ...,\ ga_kg^{-1}) \tag{1.1}
\end{equation*}
for all $a_1,\ a_2,\ ,\ ...,\ a_k$ in ${\mathcal B}$ and every invertible $g$ in ${\mathcal B}$. An invariant $1$-linear functional is usually called a trace. If $F_1$ and $F_2$ are invariant $k$-linear, and respectively, $m$-linear functionals, then an associative product $F_1\times F_2$ can be defined by
\[(F_1\times F_2)(a_1,\ a_2,\ ...,\ a_{k+m})=F_1(a_1,\ a_2,\ ...,\ a_k)F_2(a_{k+1},\ a_{k+2},\ ,\ ...,\ a_{k+m}).\]
Clearly, $F_1\times F_2$ is an invariant $(k+m)$-linear functional. We let ${\mathcal F}^0=\cc$, and ${\mathcal F}^k$ be the space of 
invariant $k$-linear functionals on ${\mathcal B}$, $k\geq 1$, and set 
\[{\mathcal F}^*({\mathcal B})=\oplus_{k=0}^{\infty}{\mathcal F}^k({\mathcal B}).\]
Then ${\mathcal F}^*$ is a graded algebra over $\cc$ with respect to above-defined product. Now consider the map $\tau$ from ${\mathcal F}^*$ to holomorphic forms on $P^c(A)$ defined by
\[\tau(1)=1,\ {\text and}\ \tau (F)=F(\omega_A(z),\ \omega_A(z),\ ...,\ \omega_A(z)).\]
\begin{prop}
$\tau$ a homomorphism from ${\mathcal F}^*({\mathcal B})$ into $H^*_d(P^c(A),\ \cc)$.
\end{prop}

\begin{exam}
Let $A$ be a $n$-tuple of elements from a commutative Banach algebra ${\mathcal B}$. A bounded linear functional $\phi$ is said to be multiplicative if $\phi(ab)=\phi(a)\phi(b)$, $\forall a,\ b\in {\mathcal B}$. The collection of multiplicative linear functionals on ${\mathcal B}$ is often called the maximal ideal space of ${\mathcal B}$, and we denote it by $M({\mathcal B})$.
Then by Gelfand's theorem $A(z)$ is not invertible in ${\mathcal B}$ if and only if there exists a multiplicative linear functional $\phi$ such that 
\begin{equation*}
\phi(A(z))=\sum_{j=1}^{n}z_j\phi(A_j)=0.\tag{1.2}
\end{equation*}
For simplicity, we let $H_\phi=\{z\in {\cc}^n:\ \sum_{j=1}^{n}z_j\phi(A_j)=0\}.$ 
If $\phi$ is such that $\phi(A_j)=0$ for all $j$, then $P(A)={\cc}^n$. Otherwise $H_\phi$ is a hyperplane, and one sees that
\begin{equation*}
P(A)=\cup_{\phi\in M({\mathcal B})}H_{\phi}.\tag{1.3}
\end{equation*} 

In general $M({\mathcal B})$ is an infinite set. In the case when $P(A)$ is a union of a finite number 
of hyperplanes, $P(A)$ is also called a {\em central hyperplane arrangement} (cf. [OT]). One sees that every multiplicative linear 
functional $\phi$ is a trace, hence is in ${\mathcal F}^1$ with $\phi(I)\neq 0$. By Lemma 3.3 in [Ya],
$\phi(\omega_A(z))$ is a nontrivial element in $H^1_d(P^c(A),\ \cc)$. Furthermore, on $P^c(A)$, 
 \begin{align*}
\phi(\omega_A(z))&=\phi(A^{-1}(z))d\phi(A(z))\\
&=\frac{d\phi(A(z))}{\phi(A(z))}\\
&=\frac{d\sum_{j=1}^n z_j\phi(A_j)}{\sum_{j=1}^n z_j\phi(A_j)}.\\
\end{align*}
Here one recalls that $\sum_{j=1}^n z_j\phi(A_j)$ is the defining function for the hyperplane $H_{\phi}$.
If $\phi_1,\ \phi_2,\ ...,\ \phi_s$ are multiplicative linear functionals, and 
\[F(a_1,\ a_2,\ ...,\ a_s)=\prod_{j=1}^s\phi_j(a_j),\]
then it is easy to see that $F$ belongs in ${\mathcal F}^s$, and 
\[F(\omega_A(z),\ ...,\ \omega_A(z))=\prod_{j=1}^s\phi_j(\omega_A(z)).\]

It is a classical theorem by Arnold and Brieskon [OT] that when ${\mathcal B}$ is a commutative matrix algebra, $\tau:{\mathcal F}^* \longrightarrow H^*_d(P^c(A),\ \cc)$ is surjective. 
\end{exam}

Example 1.2 shows that in the commutative case, invariant multi-linear functionals are very effective tools in extracting topological information about $P^c(A)$ from $\omega_A$. Unfortunately, for many important non-commutative algebras, there are no or very few nontrivial invariant linear functionals.

\section{A cohomological map}

This section starts to look at the non-commutative case, with an aim to establish a map from
the cyclic cohomology of ${\mathcal B}$ to the de Rham cohomology of $P^c(A)$.
A good resource for the following material can be found in [Co].
Let $\mathcal{B}$ be a topological algebra over $\cc$ that is associative but not necessarily commutative nor unital.

\begin{defn}
The Hochschild complex of $\mathcal{B}$ is the cochain complex $C^p(\mathcal{B})$ of continuous $(p+1)$-linear functionals on $\mathcal{B}$ with respect to the coboundary map $b_{p-1}:\ C^{p-1}(\mathcal{B})\longrightarrow C^{p}(\mathcal{B})$ defined by $b_{-1}=0$ and for $p\geq 1$
\begin{align*}
&(b_p\phi)(\alpha_1,\alpha_2,\dots,\alpha_{p+1})\\
&=\sum_{j=1}^{p}(-1)^{j-1}\phi(\alpha_1,\dots,\alpha_j\alpha_{j+1}, \dots,\alpha_{p+1})+(-1)^{p}\phi(\alpha_{p+1}\alpha_1,\dots,\alpha_p).
\end{align*}
We will simply denote all the coboundary maps by $b$ when there is no danger of confusion. The $p$-th Hochschild cohomology space is,
\[HH^p({\mathcal A}): =ker(b_p)/im(b_{p-1}), \ p\geq 0.\]
\end{defn}

\begin{exam}
$HH^0({\mathcal A}) =\text{ker} (b : C^{0}(\mathcal{B}) \to C^{1}(\mathcal{B}))=\{ \text{traces on} \ {\mathcal B}\}.$
\end{exam} 

\begin{defn}
 A continuous Hochschild cochain $\phi \in C^p(\mathcal{B})$ is cyclic if 
$$\phi(a_0,a_1,\dots,a_{p})=(-1)^{p}\phi(a_{p},a_0,a_1,\dots,a_{p-1}). $$  We will denote the space of all continuous $p$-dim cyclic cochains on $\mathcal{B}$ as $C^{p}_{\lambda}(\mathcal{B})$
\end{defn}

An observation made by Connes is that if $\phi$ is a cyclic cochain then $b\phi$ is a cyclic cochain. Hence we can define a complex $C^{*}_{\lambda}$ with $b$ as the coboundary map.  

\begin{defn}
 The cyclic complex of $\mathcal{B}$ is the cochain complex of $C^*_{\lambda}(\mathcal{B})$ with respect to the coboundary map $b.$  The $pth$ cohomology space is denoted $HC^p(\mathcal{B})$ and called the $p$-th cyclic cohomology of $\mathcal{B}$. 
\end{defn}
Clearly, $HC^0({\mathcal A})=HH^0({\mathcal A})$.
  
In what follows, we will establish a map from $HC^p(\mathcal{B})$ to the de Rham cohomology $H^{p+1}_d(P^c(A),\ \cc)$. Since much of the work does not rely on the linearity of $A(z)$, we proceed in a more general setting. Let $f$ be an entire function on $\cc^n$ with values in a unital topological algebra $\mathcal{B}.$ Define 
\begin{equation*}
\sigma (f):=\{z\in \cc^n:\ f(z)\ \text{ is not invertible in ${\mathcal B}$}\}.
\end{equation*}%
$\sigma (f)$ can be equal to $\cc^n$ in some cases, but throughout this paper we assume the resolvent set $\sigma^{c}(f)$ is nonempty.   
Then $\omega_{f}(z):=f(z)^{-1}df(z)$ is well-defined and holomorphic on $\sigma^{c}(f).$  The identity
\begin{equation*}
 d\omega_{f}(z)=-\omega_{f}(z)\wedge \omega_{f}(z),\tag{2.1}
 \end{equation*}
is not difficult to check (cf. [BCY]), and it is the key to the upcoming work.

The goal here is to use the operator-valued $1$-form $\omega_{f}(z)$ to construct a map between the cyclic cohomology and the de Rham cohomology.  

\begin{defn}
Let $f$ be an entire function with values in a unital topological algebra $\mathcal{B}.$  Then we define the linear map $\kappa: C_{\lambda}^{p}(\mathcal{B}) \to \Omega^{p+1}(\sfc)$ by 
$$\kappa(\phi)=\phi(\omega_{f},\omega_f,\dots,\omega_f).$$ 
\end{defn}
Here, $\Omega^{q}(\sfc)$ stands for the collection of smooth $q$-forms on $\sfc$. We are looking to establish that the map $\kappa$ will descend to the cohomology level, i.e. $$\kappa: HC^{p}(\mathcal{B}) \to H_{d}^{p+1}(\sfc,\ \cc).$$ In fact, we will show that, up to a constant dependent on $p$, the diagram
$$
\begin{array}{ccc}
 C_{\lambda}^{p}(\mathcal{B}) & \stackrel{\kappa}{\longrightarrow} & \Omega^{p+1}(\sfc)\\
 \downarrow b & & \downarrow -d\\
  C_{\lambda}^{p+1}(\mathcal{B}) & \stackrel{\kappa}{\longrightarrow} & \Omega^{p+2}(\sfc),
\end{array}
$$
commutes. The proof is carried out in the following two lemmas.

First, it is worth noting that the space of holomorphic forms is generated by $dz_{i_1}\wedge dz_{i_2} \wedge \cdots \wedge dz_{i_r}$ for $1\leq r \leq n$.  Thus, if $f(z)$ is an entire function from $\mathbb{C}^{n}$ to $\mathcal{B}$ then we may write $\omega_f=\sum_{i=1}^{n}B_idz_i$, where 
$B_i(z)=f^{-1}(z)\frac{\partial f}{\partial z_i}$. In order to establish this map we make the following observation.
\begin{lem}
Let $f$ be an entire function with values in a unital topological algebra $\mathcal{B}.$  
If $\phi$ is a cyclic $p$-linear functional on  $\mathcal{B}$ then 
$$\kappa(b\phi)=-d(\kappa(\phi))-\phi(d\omega_f,\omega_f,\dots,\omega_f)$$
\end{lem}
\begin{proof}
Write $\omega_f=\sum_{i=1}^{n}B_idz_i$ as noted.
For $1\leq p\leq n$, we let $\{G_p^i:\ i=1,\ 2,\ ...,\ \binom{n}{p}\}$ be the collection of all subsets of $\{1,\ 2,\ ...,\ n\}$ of size $p$, and let $S(p,i)$ be the group of permutations on $G_p^i$.
Then 
\begin{align*}
&\kappa(\phi)= \phi(\omega_f,\omega_f,\dots,\omega_f)\\
&=\sum_{i=1}^{\binom{n}{p}}\sum_{\pi \in S(p,i)}sgn(\pi)\phi(B_{\pi(j_1)},B_{\pi(j_2)},\dots, B_{\pi(j_p)})dz^{i},
\end{align*}
where $dz^{i}=dz_{j_1}dz_{j_2}\cdots dz_{j_p}$, and $j_1,\ j_2,\ ...,\ j_p$ are the elements in $G_p^i$ in increasing order.
Now recall that $b\phi$ is a $(p+1)$-linear functional defined by 
\begin{align*}
&(b\phi)(\alpha_1,\alpha_2,\dots,\alpha_{p+1})\\
&=\sum_{j=1}^{p}(-1)^{j-1}\phi(\alpha_1,\dots,\alpha_j\alpha_{j+1}, \dots,\alpha_{p+1})+(-1)^{p}\phi(\alpha_{p+1}\alpha_1,\dots,\alpha_p).
\end{align*}
Thus we have that
\begin{align*}
&\kappa(b\phi)\\
=&b\phi(\omega_f,\omega_f,\dots,\omega_f)\\
=&\sum_{i=1}^{\binom{n}{p+1}}\sum_{\pi \in S(p+1,\ i)}sgn(\pi)b\phi(B_{\pi(j_1)},B_{\pi(j_2)},\dots, B_{\pi(j_{p+1})})dz^i\\
=&\sum_{i=1}^{\binom{n}{p+1}}\sum_{\pi \in S(p+1,\ i)}sgn(\pi)\big(\sum_{s=1}^{p}(-1)^{s-1}\phi(B_{\pi(j_1)},\dots,B_{\pi(j_s)}B_{\pi(j_{s+1})} ,\dots, B_{\pi(j_{p+1})})\\
&+(-1)^{p}\phi(B_{\pi(j_{p+1})}B_{\pi(j_1)},\dots, B_{\pi(j_p)})\big)dz^i.
\end{align*}

Now we make the following observation: if we replace $\omega_{f}$ in the $sth$ coordinate by $d\omega_f=-\omega_f \wedge \omega_f$ we get
\begin{align*}
&\phi(\omega_f,\dots,d\omega_f, \dots, \omega_f)\\
&=-\phi(\omega_f,\dots,\omega_f \wedge \omega_f, \dots, \omega_f)\\
&=-\sum_{i=1}^{\binom{n}{p+1}} \sum_{\pi \in S(p+1,i)}sgn(\pi)\phi(B_{\pi(j_1)},\dots,B_{\pi(j_s)}B_{\pi(j_{s+1})} ,\dots,B_{\pi(j_{p+1})})dz^i.
\end{align*}
Also observe that for any continuous $p$-linear functional $\theta$ we have
$$ d\theta(\omega_f,\omega_f,\dots,\omega_f)=\sum_{s=1}^{p}(-1)^{s-1}\theta(\omega_f,\dots, d\omega_f,\dots,\omega_f).$$
Therefore,
\begin{align*}
&-d(\kappa(\phi))\\
&=-\sum_{s=1}^{p}(-1)^{s-1}\phi(\omega_f,\dots,d\omega_f, \dots, \omega_f)\\
&=\sum_{s=1}^{p}(-1)^{s-1}\phi(\omega_f,\dots,\omega_f \wedge \omega_f, \dots, \omega_f)\\
&=\sum_{i=1}^{\binom{n}{p+1}} \sum_{\pi \in S(p+1,\ i)}sgn(\pi)\big(\sum_{s=1}^{p}(-1)^{s-1}\phi(B_{\pi(j_1)},\dots,B_{\pi(j_s)}B_{\pi(j_{s+1})} ,\dots,B_{\pi(j_{p+1})}\big)dz^i.
\end{align*}
Hence 
\begin{align*}
&\kappa(b\phi)+d(\kappa(\phi))\\
&=(-1)^{p}\sum_{i=1}^{\binom{n}{p+1}}\sum_{\pi \in S(p+1,\ i)}sgn(\pi)\phi(B_{\pi(j_{p+1})}B_{\pi(j_1)},\dots, B_{\pi(j_p)})dz^i.
\end{align*}
Now we need to compare the right-hand side with $\phi(d\omega_f,\omega_f,\dots,\omega_f)$. Let $r$ be the action on $S_{p+1}$ defined by 
$r\pi(j)=\pi(j+1),\ 1\leq j\leq p$ and $r\pi(p+1)=\pi(1)$. Clearly $rS_{p+1}=S_{p+1}$, and $sgn(r^{-1}\pi)=(-1)^p sgn(\pi)$. Therefore
\begin{align*}
&\sum_{i=1}^{\binom{n}{p+1}}\sum_{\pi \in S(p+1,\ i)}sgn(\pi)\phi(B_{\pi(j_{p+1})}B_{\pi(j_1)},\dots, B_{\pi(j_p)})dz^i\\
&=(-1)^{p}\sum_{i=1}^{\binom{n}{p+1}}\sum_{\pi \in S(p+1,\ i)}sgn(r^{-1}\pi)\phi(B_{r^{-1}\pi(j_{1})}B_{r^{-1}\pi(j_2)},\dots, B_{r^{-1}\pi(j_{p+1})})dz^i.\\
&=(-1)^{p}\sum_{i=1}^{\binom{n}{p+1}}\sum_{\pi \in S(p+1,\ i)}sgn(\pi)\phi(B_{\pi(j_{1})}B_{\pi(j_2)},\dots, B_{\pi(j_{p+1})})dz^i\\
&=(-1)^{p}\phi(\omega_f\wedge\omega_f,\omega_f,\dots,\omega_f).
\end{align*}

Thus we have that
\begin{align*}
&\kappa(b\phi)+d(\kappa(\phi))\\
&=\phi(\omega_f\wedge\omega_f,\omega_f,\dots,\omega_f)\\
&=-\phi(d\omega_f,\omega_f,\dots,\omega_f). 
\end{align*}
\end{proof}

\begin{lem}
 If $\phi$ is a cyclic $p$-linear functional for $p\geq 1$ then
$$\phi(d\omega_f,\omega_f,\dots,\omega_f)=-\frac{1}{p+1}\kappa (b\phi).$$  
\end{lem}

The proof is straight forward calculation. But one needs to take good care of indices. To see the nature of the proof, we explicitly compute a simple case of Lemma 2.7. 
Assume $\omega_f=\sum_{i=1}^4B_i(z)dz_i$ and $\phi$ a cyclic $2$-linear functional. 
Because of the symmetry it is sufficient to display the work on the $ dz_1 \wedge dz_2 \wedge dz_3$ term. Thus
\begin{align*}
&-\phi(d\omega_A,\omega_A)\\
&= \phi(\omega_A \wedge \omega_A,\omega_A)\\
&=\phi((\sum_{i=1}^{4}B_idz_i)\wedge (\sum_{i=1}^{4}B_idz_i),(\sum_{i=1}^{4}B_idz_i))\\
&=\big(\phi(B_1B_2-B_2B_1,B_3)+\phi(B_3B_1-B_1B_3,B_2)\\
&+\phi(B_2B_3-B_3B_2,B_1)\big)dz_1 \wedge dz_2 \wedge dz_3\\
&+\cdots ,
\end{align*}
and using the cyclicity of $\phi$ one verifies that
\begin{align*}
&\phi(B_1B_2-B_2B_1,B_3)+\phi(B_3B_1-B_1B_3,B_2)+\phi(B_2B_3-B_3B_2,B_1)\\
=&\phi(B_1B_2,B_3)-\phi(B_1,B_2B_3)+\phi(B_3B_1,B_2)\\
&-\phi(B_1B_3,B_2)+\phi(B_1,B_3B_2)-\phi(B_2B_1,B_3)\\
=&b\phi(B_1,B_2,B_3)-b\phi(B_1,B_3,B_2),
\end{align*}
which is $1/3$ of the coefficient of $dz_1dz_2dz_3$ in $\kappa(b\phi)$.
\begin{proof}
As in the example above, we only display the work on the $ dz_1 \wedge dz_2 \wedge \cdots \wedge dz_{p+1}$ (which we denote by $dz^1$) term.
Write
\begin{align*}
&-\phi(d\omega_f,\omega_f,\dots, \omega_f)\\
=&\phi(\omega_f \wedge \omega_f, \omega_f,\ \dots, \omega_f)\\
=&\sum_{\pi \in S_{p+1}}sgn(\pi)\phi(B_{\pi(1)}B_{\pi(2)},\ B_{\pi(2)},\ \dots,\ B_{\pi({p+1})})dz^1\\
&+\cdots.
\end{align*}
Observe that $G=\{r,\ r^2,\ ...,\ e=r^{p+1}\}$ is a cyclic
group acting on $S_{p+1}$. We decompose $S_{p+1}$ as a disjoint union of $G$-orbits $\cup_{s=1}^{p!}G\pi_s$, and compute
\begin{equation*}
\sum_{\pi \in G\pi_s}sgn(\pi)\phi(B_{\pi(1)}B_{\pi(2)},\ B_{\pi(2)},\ \dots,\ B_{\pi({p+1})}). \tag{2.2}
\end{equation*}
By symmetry, it is sufficient to consider the case for $G\pi_1$, and we write $\pi_1$ as $\pi$ for simplicity. Then the summation (2.2) over $G\pi_1$ is equal to
\begin{align*}
& sgn(\pi)\phi(B_{\pi(1)}B_{\pi(2)},\ B_{\pi(3)},\ \dots,\ B_{\pi({p+1})})\\
&+sgn(r\pi)\phi(B_{\pi(2)}B_{\pi(3)},\ B_{\pi(4)},\ \dots,\ B_{\pi({1})})\\
&+sgn(r^2\pi)\phi(B_{\pi(3)}B_{\pi(4)},\ B_{\pi(5)},\ \dots,\ B_{\pi({2})})\\
&\cdots \\
&+sgn(r^p\pi)\phi(B_{\pi(p+1)}B_{\pi(1)},\ B_{\pi(2)},\ \dots,\ B_{\pi({p})})\\
=&sgn(\pi)\{\phi(B_{\pi(1)}B_{\pi(2)},\ B_{\pi(2)},\ \dots,\ B_{\pi({p+1})})\\
&+(-1)^p\phi(B_{\pi(2)}B_{\pi(3)},\ B_{\pi(4)},\ \dots,\ B_{\pi({1})})\\
&+(-1)^{2p}\phi(B_{\pi(3)}B_{\pi(4)},\ B_{\pi(5)},\ \dots,\ B_{\pi({2})})\\
&\cdots \\
&+(-1)^{p^2}\phi(B_{\pi(p+1)}B_{\pi(1)},\ B_{\pi(2)},\ \dots,\ B_{\pi({p})})\}
\end{align*}
Then by the cyclicity of $\phi$, the quantity above is equal to
\begin{align*}
&sgn(\pi)\{\phi(B_{\pi(1)}B_{\pi(2)},\ B_{\pi(3)},\ \dots,\ B_{\pi({p+1})})\\
=&+(-1)^{p+p-1}\phi(B_{\pi({1})},\ B_{\pi(2)}B_{\pi(3)},\ B_{\pi(4)},\ \dots,\ B_{\pi({p+1})})\\
&+(-1)^{2p+2(p-1)}\phi(B_{\pi({1})},\ B_{\pi(2)},\ B_{\pi(3)}B_{\pi(4)},\ \dots,\ B_{\pi({p+1})})\\
&\cdots \\
&+(-1)^{(p-1)p+(p-1)(p-1)}\phi(B_{\pi({1})},\ B_{\pi(2)},\ \dots,\ B_{\pi({p})}B_{\pi({p+1})})\\
&+(-1)^{p^2-p+p}\phi(B_{\pi(p+1)}B_{\pi(1)},\ B_{\pi(2)},\ \dots,\ B_{\pi({p})})\}.\\
=&sgn(\pi)b\phi(B_{\pi(1)},\ \dots,\ B_{\pi(p+1)}).
\end{align*}
Hence, with respect to the decomposition
\[S_{p+1}=\cup_{s=1}^{p!}G\pi_s,\]
the coefficient of $ dz_1 \wedge dz_2 \wedge \cdots \wedge dz_{p+1}$ in $-\phi(d\omega_f,\omega_f,\dots,\omega_f)$ is 
\[\sum_{s=1}^{p!}sgn(\pi_s)b\phi(B_{\pi_s(1)},\ B_{\pi_s(2)},\ ...,\ B_{\pi_s(p+1)}).\]
Now let us look at the coefficient of $dz^1$ in $\kappa(b\phi)$. One sees that
\begin{align*}
&\kappa(b\phi)\\
&=b\phi(\omega_f,\ ...,\ \omega_f)\\
&=\sum_{s=1}^{p!}\sum_{\pi\in G\pi_s}sgn(\pi)b\phi(B_{\pi(1)},\ ...,\ B_{\pi(p+1)})dz^1\\
&+ \cdots.
\end{align*}

Since $b\phi$ is a cyclic $(p+1)$-linear functional and $sgn(r\pi_s)=(-1)^p sgn(\pi_s)$,
\begin{align*}
& sgn(r\pi_s)b\phi(B_{\pi_s(2)},\ B_{\pi_s(3)},\ ...,\ B_{\pi_s(1)})\\
&=sgn(\pi_s)b\phi(B_{\pi(1)},\ B_{\pi(2)},\ ...,\ B_{\pi(p+1)}),
\end{align*}
and it follows that 
\begin{align*}
&\sum_{\pi\in G\pi_s} sgn(\pi)b\phi(B_{\pi(1)},\ B_{\pi(2)},\ ...,\ B_{\pi(p+1)})\\
&=(p+1)sgn(\pi_s)b\phi(B_{\pi_s(1)},\ B_{\pi_s(2)},\ ...,\ B_{\pi_s(p+1)}),
\end{align*}
i.e. the coefficient of $ dz_1 \wedge dz_2 \wedge \cdots \wedge dz_{p+1}$ in $-\phi(d\omega_f,\omega_f,\dots,\omega_f)$ is $\frac{1}{p+1}$ times that in 
$\kappa (b\phi)$. By symmetry we conclude that
\begin{equation*}
\phi(d\omega_f,\omega_f,\dots,\omega_f)=\frac{-1}{p+1}\kappa (b\phi). \tag{2.3}
\end{equation*}.
\end{proof}
Combining Lemmas 2.7 and 2.8, we have
\begin{thm}
If $\phi$ is a continuous cyclic $p$-linear functional, then
\[\frac{p}{p+1}\kappa (b\phi)=-d\kappa(\phi).\]
\end{thm}

Clearly, the map $\kappa$ takes closed cyclic cochains to closed forms and exact cyclic cochains to exact forms.
Hence we have
\begin{corr}
The map $\kappa$ is a homomorphism from $HC^{p}(\mathcal{B})$ into $H_{d}^{p+1}(\sfc,\ \cc)$,\ $p\geq 0$.   
\end{corr}
 
\begin{exam}
Let $\mathcal{B}$ be the matrix algebra $M_k(\cc)$. Then $HC^0(\mathcal{B})=\cc tr$, where $tr$ is the classical trace on square matrices. Then for any analytic function $f:\ {\cc}^n\to \mathcal{B}$, $\sigma_f=\{z:\ detf(z)= 0\}$, which is a hypersurface. And on $\sfc$
\begin{equation*}
tr(\omega_f)=dlog(detf(z)).\tag{2.4}
\end{equation*}

\end{exam}

\section{A high-order Jacobi's formula}

Formula 2.4 is known as the classical Jacobi's formula, and it is very useful. In this section, we will seek higher order generalizations of this formula. We consider the case when ${\mathcal B}$ is a topological algebra with a trace $\phi$. It is easy to see that if $k$ is odd, then 
\begin{equation*}
F(a_{1},\ a_{2},\ ...,\ a_{k}):=\phi(a_{1}a_{2}\cdots a_{k})
\end{equation*}%
is a cyclic $(k-1)$-cocycle. Hence by Corollary 2.10
\begin{equation*}
F(\omega _{f}(z),\ ...,\ \omega _{f}(z))=\phi(\omega _{f}^{k}(z))
\end{equation*}%
is a closed $k$-form on $\sigma ^{c}(f)$. It is worth noting that $\phi(\omega _{f}^{k}(z))=0$ when $k$ is even (cf. [BCY]). So, is there a closed form formula for 
$\phi(\omega _{f}^{k}(z))$ when $k$ is odd? When $f:\ \cc^n\to {\mathcal B}$ is homogeneous and $k=n-1$, the answer is yes.
We begin our study with the following lemma. Here, we recall some notational conventions.  We write 
$\omega_f=\sum_{s=1}^nW_s(z)dz_s$, where $W_s(z)=f^{-1}(z)\frac{\partial f}{\partial z_s}$. Recall that for $1\leq p\leq n$, we let $\{G_p^j:\ j=1,\ 2,\ ...,\  \binom{n}{p}\}$ be the collection of all subsets of $\{1,\ 2,\ ...,\ n\}$ of size $p$, and let $S(p,j)$ be the group of permutations on $G_p^j$. Assume $j_1,\ j_2,\ ...,\ j_p$ are the elements in $G_p^j$ in increasing order. We let $s(p,j)$ denote the permutation group of the subset $\{j_2,j_3,\dots, j_p\}.$ For a $\pi\in s(p,j)$, we write $W_{\pi(j_2)}W_{\pi(j_3)}\dots W_{\pi(j_p)}$ simply as $W_{\pi}$. For convenience, we write $dz_{j_1}dz_{j_2}\cdots dz_{j_p}$ as $dz^j$.

\begin{lem}\label{lem:trace}
Let ${\mathcal B}$ be a topological algebra with a continuous trace $\phi.$ Then 
$$\phi(\omega_f^{p})=p\sum_{j=1}^{\binom{n}{p}}\sum_{\pi \in s(p,j)}sgn (\pi)\phi(W_{j_1}W_{\pi})dz^j.$$    
\end{lem}

\begin{proof}
First,
\begin{align*}
&\phi(\omega^p_f)=\phi( \omega_f\wedge\omega_f\wedge\dots\wedge\omega_f)\\
&=\sum_{j=1}^{\binom{n}{p}}\sum_{\pi \in S(p,j)}sgn(\pi)\phi(W_{\pi(j_1)}W_{\pi(j_2)}\dots W_{\pi(j_p)})dz^{j}.
\end{align*}
Recall that $r$ is the cyclic action on the permutation group $S_{p}$ defined by 
$r\pi(j)=\pi(j+1),\ 1\leq j\leq p-1$ and $r\pi(p)=\pi(1)$. Clearly $rS_{p}=S_{p}$, and $sgn(r^{-1}\pi)=(-1)^{p-1} sgn(\pi)$. Since $p$ is odd,
$sgn(r^{-1}\pi)=sgn(\pi)$. Since $\phi$ is a trace, $\phi (W_{r^{-1}\pi})=\phi (W_{\pi})$. Moving $W_{j_1}$ to the first in every product, we have
\begin{align*}
 &\sum_{\pi \in S(p,j)}sgn(\pi)\phi(W_{\pi(j_1)}W_{\pi(j_2)}\dots W_{\pi(j_p)})\\
&=p\sum_{\pi \in s(p,j)}sgn(\pi)\phi(W_{j_1}{W_{\pi}}).\tag{3.1}
 \end{align*}

\end{proof}

We now consider the case $f$ is homogeneous. There is a closed form formula for $\phi(\omega^{n-1}_f)$
in this case. For $1\leq j\leq n$, we let $G_{n-1}^j:=\{1,\ 2,\ ...,\ j-1,\ j+1,\ ...,\ n\}$, and write $dz_{1}dz_{2}\cdots dz_{j-1}dz_{j+1}\cdots dz_{n}$ as $dz_{\bar j}$.

\begin{thm} \label{thm:trace}
For $n$ even, if $f$ is a homogeneous holomorphic function from $\cc^n$ to a topological algebra ${\mathcal B}$ with a continuous trace $\phi$, then
\[\phi(\omega_{f}^{n-1})=q(z)s(z)\] for some holomorphic function $q(z)$ on $\sigma^c(f)$ and \[s(z)=\sum_{j=1}^{n}(-1)^{j}z_j\cdot dz_{\bar{j}}.\] 
\end{thm}
We note here that it will be helpful to read Example 3.6 before reading the following proof.
\begin{proof}
Assume $f(\lambda z)=\lambda ^m f(z)$ for some $m\in {\mathbb Z}$ and all non-zero complex numbers $\lambda$.
By evaluating $\frac{df(\lambda z)}{d\lambda}$ at $\lambda=1$ we get
\begin{equation*}
\sum_{k=1}^nz_k\frac{\partial f(z)}{\partial z_k}=mf(z).\tag{3.2}
\end{equation*}
Notice, when $p=n-1$, we let $G_{n-1}^j=\{\{1,\ 2,\ ...,\ n\}\setminus \{j\};\ j=1,\ 2,\ ...,\ n.\}$
First, observe that $j_1=1$ for $2\leq j\leq n$ and $1_1=2$, and
the formula in Lemma 3.1 in this case has the form 
\begin{align*}
\phi(\omega_f^{n-1})&=(n-1)\big(\sum_{\pi \in s(n-1,1)}sgn (\pi)\phi(W_{2}W_{\pi})dz_{\bar{1}}\\
&+\sum_{j=2}^n\sum_{\pi \in s(n-1,j)}sgn (\pi)\phi(W_{1}W_{\pi})dz_{\bar{j}}\big). \tag{3.3}
\end{align*}
For simplicity, we denote the right-hand side by $(n-1)\sum_{j=1}^{n}I_{\bar{j}} dz_{\bar{j}}$. 

Now we determine how $I_{\bar{j}}$ is related to $I_{\bar{n}}$ for each $j$. By symmetry, it is sufficient to display the work for $I_{\bar{1}}$. 
Observe that $s(n-1,n)$ is the permutations of the $n-2$ elements $\{2,3,\dots,n-1\}.$ To begin, we claim that for $2\leq i\leq n-1$,
\begin{equation*}
\sum_{\pi \in s(n-1,n)}sgn (\pi)\phi(W_iW_{\pi})=0.\tag{3.4}
\end{equation*}
Let $\pi\in s(n-1,n)$ be such that $\pi(k)=i$ for some $2\leq k\leq n-1$. Then since $\phi$ is a trace,
\begin{align*}
&\phi(W_{\pi(k)}W_{\pi(2)}W_{\pi(3)}\dots W_{\pi(k)}W_{\pi(k+1)}\dots W_{\pi(n-1)})\\
&=\phi(W_{\pi(k)}W_{\pi(k+1)}\dots W_{\pi(n-1)}W_{\pi(k)}W_{\pi(2)}\dots W_{\pi(k-1)}).
\end{align*}
If $\pi'\in s(n-1,n)$ is such that \[\pi'\big(2,\ 3,\ ...,\ n-1\big)=\big(\pi(k+1),\ ...,\ \pi(n-1),\ \pi(k),\ \pi(2),\ ...,\ \pi(k-1)\big),\]
then $sgn(\pi')=(-1)^{k-2}(-1)^{(k-1)(n-k-1)}sgn(\pi)$ (first move the middle $\pi(k)$ to the end then transplant $(\pi(k+1),\ ...,\ \pi(n-1))$ to the right of $\pi(k)$). Since $n$ is even, $(-1)^{k-2}(-1)^{(k-1)(n-k-1)}=-1$, hence each summand in (3.4) is paired off. Thus we have (3.4). Now one checks that
\begin{align*}
&z_1\cdot I_{\bar{n}}\\
=&z_1\sum_{\pi \in s(n-1,n)}sgn (\pi)\phi(W_{1}W_{\pi})\\
=&\sum_{\pi \in s(n-1,n)}sgn (\pi)\phi(z_1W_1W_{\pi}) \\
=&\sum_{\pi \in s(n-1,n)}sgn (\pi)\phi(z_1W_1W_{\pi})+\sum_{\pi \in s(n-1,n)}sgn (\pi)\phi(z_2W_2W_{\pi})\\
&+\sum_{\pi \in s(n-1,n)}\phi(z_3W_3W_{\pi})+\cdots +\sum_{\pi \in s(n-1,n)}\phi(z_{n-1}W_{n-1}W_{\pi})\\
&+\sum_{\pi \in s(n-1,n)}sgn (\pi)\phi(z_nW_nW_{\pi})-\sum_{\pi \in s(n-1,n)}sgn (\pi)\phi(z_nW_nW_{\pi})\\
=&\sum_{\pi \in s(n-1,n)}sgn (\pi)\phi((\sum_{k=1}^{n}z_kW_k)W_{\pi})-\sum_{\pi \in s(n-1,n)}sgn (\pi)\phi(z_nW_nW_{\pi}). \tag{3.5}
\end{align*}
By (3.2), \[(\sum_{k=1}^{n}z_kW_k)W_{\pi}=\big(f^{-1}(z)\sum_{k=1}^nz_k\frac{\partial f(z)}{\partial z_k}\big)W_{\pi}=mW_{\pi}.\]
Observe that $\sum_{\pi \in s(n-1,n)}sgn (\pi)\phi(W_{\pi})$ is the coefficient of $dz_2dz_3\dots dz_{n-1}$ in $\phi(\omega^{n-2})$, and it was remarked earlier that
$\phi(\omega^{k})=0$ for every even $k$. Hence the first term in the right-hand side of (3.5) is $0$. To study the second term, for convenience we first let $G$ denote the permutation group on $\{2,\ 3,\ ...,\ n\}$ (i.e. $G=S(n-1,\ 1)$), then by (3.1),
\[I_{\bar 1}=\frac{1}{n-1}\sum_{\pi\in G}sgn(\pi)\phi(W_{\pi}).\]
If $\pi\in G$ is such that $\pi(k)=n$, then we have 
\begin{align*}
\phi(W_{\pi})&=\phi(W_{\pi(2)}W_{\pi(3)}\dots W_{\pi(k)}W_{\pi(k+1)}\dots W_{\pi(n)})\\
&=\phi(W_{\pi(k+1)}\dots W_{\pi(n)}W_{\pi(2)}\dots W_{\pi(k)}).
\end{align*}
If $\pi'\in G$ is such that 
\[\pi'(2,3,...,n)=(\pi(k+1),...,\pi(n),\pi(2),...,\pi(k)),\]
then since $n$ is even, \[sgn(\pi')=(-1)^{(k-1)(n-k)}sgn(\pi)=sgn(\pi).\]
Hence 
\begin{align*}
&\sum_{\pi\in G}sgn(\pi)\phi(W_{\pi})\\
&=(n-1)\sum_{\pi\in G,\ \pi(n)=n }sgn(\pi)\phi(W_{\pi})\\
&=(n-1)\sum_{\pi\in s(n-1,n)}sgn(\pi)\phi(W_{\pi}W_n)\\
&=(n-1)\sum_{\pi\in s(n-1,n)}sgn(\pi)\phi(W_nW_{\pi}),
\end{align*}
and it follows that $z_1\cdot I_{\bar{n}}=-z_n\cdot I_{\bar{1}}$. Similar arguments will show that 
\begin{equation*}
z_i\cdot I_{\bar{n}}=(-1)^i z_n \cdot I_{\bar{i}}\tag{3.6}
\end{equation*}
for all $1\leq i < n.$  Thus we have that 
$$\phi(\omega_A^{n-1})=(n-1)\sum_{j=1}^{n}I_{\bar{j}} dz_{\bar{j}}=(n-1)\sum_{j=1}^{n}I_{\bar{n}}\cdot \frac{z_j}{z_n}dz_{\bar{j}}.$$
If we set $s(z)=\sum_{j=1}^{n}(-1)^{j}z_j\cdot dz_{\bar{j}}$ and let $q(z)=\frac{1}{z_n}\cdot I_{\bar{n}}.$  Then we have the desired formula
$$\phi(\omega_f^{n-1})=q(z)\cdot s(z).$$ 
Further, it is clear that $I_{\bar{n}}$ is holomorphic on $\sigma^c(f)$. If $z_n=0$ then $I_{\bar{n}}=0$ by (3.6), hence $q$ is holomorphic on $\sigma^c(f)$.
\end{proof}

The connection between $f$ and $q$ in Theorem 3.2 is rather intricate. To see what is involved, we take a look at the case when ${\mathcal B}$ is the matrix algebra 
$M_{k}(\cc)$ where $k\geq 2$, and $f$ is the pencil $A(z)=z_1A_1+z_2A_2+z_3A_3+z_4A_4$. Then we have the following

\begin{thm} If $A=(A_1,A_2,A_3,A_4)$ is a tuple of elements of the Banach algebra $M_{k}(\cc),$ then on $P^c(A)$
$$
tr(\omega_A^3)=\frac{3p(z)}{det^2(A(z))}\cdot s(z)
$$
where $s(z)=-z_1dz_2\wedge dz_3\wedge dz_4 + z_2dz_1\wedge dz_3\wedge dz_4 - z_3dz_1\wedge dz_2\wedge dz_4 +z_4dz_1\wedge dz_2\wedge dz_3,$ and $p(z)$ is homogeneous of degree $2k-4.$  
\end{thm}

Before we prove this theorem we need the following lemma from Jacobi (cf. [DC]).

\begin{lem} [Jacobi]
Let $B$ be an $k \times k$ matrix and $B^{\#}$ be its adjugate matrix such that $B^{\#}B=(detB)I$. Then $$B^{\#}_{i,j}B^{\#}_{p,q}-B^{\#}_{i,q}B^{\#}_{p,j}=det(B) \cdot det(C_{(i,j),(p,q)})$$
where $C_{(i,j),(p,q)}$ is the matrix obtained from $B$ by removing rows $i,p$ and columns $j,q.$ 
\end{lem}  

\begin{proof} [Proof Of Theorem 3.3]
First we make the observation that $\omega_{A}$ is homogeneous of degree $0$.  Thus $tr(\omega_{A}^{3})$ is homogeneous of degree $0.$  Further $det(A(z))$ is homogeneous of degree $k$ and $s$ is homogeneous of degree $4.$  Thus if $tr(\omega_A^3)$ has the form in the theorem, then $p(z)$ must be a homogeneous of degree $2k-4.$  
What is left to prove is that $tr(\omega_A^3)$ has the desired form. By Lemma 3.1 we write
\begin{equation*}
\phi(\omega_A^3)=\displaystyle 3\sum_{1\leq i < j < k\leq 4} I_{ijk} dz_i\wedge dz_j\wedge dz_k, \tag{3.7}
\end{equation*} 
where $I_{ijk}= \phi(A(z)^{-1} A_iA(z)^{-1} A_j A(z)^{-1}A_k - A(z)^{-1}A_iA(z)^{-1}A_kA(z)^{-1}A_j)$. By the proof of Theorem 3.2,
we can say that $q(z)=-I_{123}/z_4.$  Hence we strive to simplify the term $I_{123}$.
First we make the following decomposition $A(z)^{-1}=\frac{1}{det(A(z))}\cdot B$ where $B = (A(z))^{\#}.$  
Thus we have the following result:
$$I_{123}=\frac{1}{(det(A(z)))^3}\cdot tr(BA_1BA_2BA_3-BA_1BA_3BA_2).$$
 
Hence what remains to be proven is that 
\begin{equation*}
tr(BA_1BA_2BA_3-BA_1BA_3BA_2)=p(z)\cdot det(A(z)). \tag{3.8}
\end{equation*}
First note that since the trace is linear we can work on elementary matrices. In particular if we denote $A_r=[a^r_{i,j}],\ r=1,\ 2,\ 3,$ and let $E_{i,j}$ to be the matrix with 1 in the $i$th row $j$th column and zero otherwise.  Then we have
\begin{align*}
&tr(BA_1BA_2BA_3-BA_1BA_3BA_2)\\
&=\sum_{1\leq p,q,k,l,r,s \leq n}a_{p,q}^{1}a_{k,l}^{2}a_{r,s}^{3}\cdot tr(BE_{p,q}BE_{k,l}BE_{r,s}-BE_{p,q}BE_{r,s}BE_{k,l}).
\end{align*}
Basic matrix multiplication leads to the fact that
$$tr(BE_{p,q}BE_{k,l}BE_{r,s}-BE_{p,q}BE_{r,s}BE_{k,l})=B_{s,p}B_{q,k}B_{l,r}-B_{l,p}B_{q,r}B_{s,k}.$$

Now we use Jacobi's formula repeatedly and have
\begin{align*}
&B_{s,p}B_{q,k}B_{l,r}-B_{l,p}B_{q,r}B_{s,k}\\
=&B_{s,p}B_{q,k}B_{l,r}-B_{q,r}(B_{l,k}B_{s,p}+det(A(z))\cdot det(C_{(l,p),(s,k)}))\\
=&B_{s,p}B_{q,k}B_{l,r}-B_{q,r}B_{l,k}B_{s,p}-B_{q,r}det(A(z))\cdot det(C_{(l,p),(s,k)})\\
=&B_{s,p}(B_{q,k}B_{l,r}-B_{q,r}B_{l,k})-B_{q,r}det(A(z))\cdot det(C_{(l,p),(s,k)})\\
=&B_{s,p}(det(A(z))\cdot det(C_{(q,k),(l,r)}))-B_{q,r}det(A(z))\cdot det(C_{(l,p),(s,k)})\\
=&det(A(z))(B_{s,p}\cdot det(C_{(q,k),(l,r)})-B_{q,r}\cdot det(C_{(l,p),(s,k)})).
\end{align*}
Hence, for any values $p,q,k,l,r,s$ we can factor out $det(A(z))$. This verifies (3.8).
Moreover,
\begin{align*}
&-z_4p(z)\\
&=det^2(A(z))I_{123}\\
&=\sum_{1\leq p,q,k,l,r,s \leq n}a_{p,q}^{1}a_{k,l}^{2}a_{r,s}^{3}\cdot (B_{s,p}\cdot det(C_{(q,k),(l,r)})-B_{q,r}\cdot det(C_{(l,p),(s,k)})).
\end{align*}

\end{proof}

There is a further reduction possible in the case $\mathcal{B}=M_2(\cc).$ 
\begin{exam}
When ${\mathcal B}= M_{2}(\cc)$ we have that the function $p(z)$ in Theorem 3.3 is a constant (which we denote by $C$). Specifically, let $A_j=\left[ \begin{array}{cc} 
a_1^j & a_2^j \\
a_3^j & a_4^j
\end{array} \right]$  for $1\leq j \leq 4.$
Then we can verify that \[C=-det \left( \begin{array}{cccc}
a_1^1 & a_1^2 & a_1^3 & a_1^4 \\
a_2^1 & a_2^2 & a_2^3 & a_2^4 \\
a_3^1 & a_3^2 & a_3^3 & a_3^4 \\
a_4^1 & a_4^2 & a_4^3 & a_4^4 
\end{array} \right).\]
So when $A_1, A_2, A_3, A_4$ are linearly independent, they form a basis for $M_2(\cc)$. The pencil $A(z)$ is then a biholomorphic map between
$P^c(A)$ and $GL(2,\ \cc)$. The pushforward of $tr(\omega^3_A)$ turns out to be a nontrivial element in the de Rham cohomology space $H_d^3(GL(2,\cc),\ \cc)$.
This observation is made in [Ya].

\end{exam}

The formula in Theorem 3.2 is not limited to the trace. Similar formulas exist for other cyclic co-cycles on ${\mathcal B}$. 

\begin{exam}
We denote by $S({\mathbb Z}^2)$ the space of sequences $(a_{n,m})_{n,m\in {\mathbb Z}^2}$ such that $(|n|+|m|)^q|a_{n,m}|$ is bounded for any $q\in {\mathbb N}$.
Let $\mathcal{A_{\theta}}$ be the algebra of which the generic element is a formal sum $\sum a_{n,m}U^{n}V^{m}$ where ($a_{n,m}) \in S(\mathbb{Z}^2)$ and the product is specified by the equality $UV=\lambda VU$ for some fixed $\lambda=exp(2\pi i\theta)$,\ $\theta\in (0,\ 1]$. We let the canonical trace, $tr$, on $\mathcal{A_{\theta}}$ be given by 
$$tr(\sum a_{n,m}U^{n}V^{m})=a_{0,0}.$$
There are two derivations $\delta_1$ and $\delta_2$ on $\mathcal{A_{\theta}}$ defined by $\delta_1(U^{m}V^{n})=mU^{m}V^{n}$ and $\delta_2(U^{m}V^{n})=nU^{m}V^{n}$. 
Connes ([Co]) showed that $HC^{1}(\mathcal{B}_{\theta})$ is a two dimensional vector space spanned by the cyclic cocycles $\phi_1$ and $\phi_{2}$ given by $\phi_{j}(x_0,x_1)=tr(x_0\delta_{j}(x_1))$ for all $x_0,\ x_1 \in \mathcal{A}_{\theta}$; and  $HC^{2}(\mathcal{A}_{\theta})$ is a two dimensional vector space spanned by
$\psi_1$ and $\psi_2$, where \[\psi_1(x_0,x_1,x_2)=tr(x_0x_1x_2)\] and 
\[\psi_2(x_0,x_1,x_2)=tr(x_0(\delta_1(x_1)\delta_2(x_2)-\delta_2(x_1)\delta_1(x_2))),\ \forall x_i\in \mathcal{A_{\theta}}.\]
Fix a topology on $\mathcal{A_{\theta}}$ such that $tr,\ \delta_1,\ \delta_2$ are continuous. Then $\phi_i$ and $\psi_{i},\ i=1,\ 2$ are all continuous multi-linear functionals. 

We first look at the map $\kappa$ on $HC^{1}$. Observe that we have that $\phi_{j}(I,C)=0$ and by the cyclicity $\phi_{j}(C,C)=0$, $j=1,\ 2,$ for all $C$ in $\mathcal{B}_{\theta}$.  
For simplicity we let $f$ be the linear pencil $A(z)$, where $A=(A_1,A_2,A_3)$ is a 3-tuple of elements in $\mathcal{A_{\theta}}$. Writing 
$\omega_A=W_1dz_1+W_2dz_2+W_3dz_3$, we have that
\begin{align*}
&\kappa(\phi_{j})=\phi_{j}(\omega_A,\omega_A) \\
=&2\phi_{j}(W_1,W_2)dz_1\wedge dz_2 +2\phi_{j}(W_1,W_3)dz_1\wedge dz_3 +2\phi_{j}(W_2,W_3)dz_2\wedge dz_3.
\end{align*}
Now we have that
\begin{align*}
&2z_1\phi_{j}(W_1,W_2)\\
=& 2z_1\phi_{j}(W_1,W_2)+2z_2\phi_{j}(W_2,W_2)+2z_3\phi_{j}(W_3,W_2)-2z_3\phi_{j}(W_3,W_2)\\
=& 2\phi_{j}(I,W_2)-2z_3\phi_{j}(W_3,W_2)\\
=&2z_3\phi_{j}(W_2,W_3);
\end{align*}
and
\begin{align*}
&2z_2\phi_{j}(W_1,W_2)\\
=& 2z_1\phi_{j}(W_1,W_1)+2z_2\phi_{j}(W_1,W_2)+2z_3\phi_{j}(W_1,W_3)-2z_3\phi_{j}(W_1,W_3)\\
=& 2\phi_{j}(W_1,I)-2z_3\phi_{j}(W_1,W_3)\\
=&-2z_3\phi_{j}(W_1,W_3).
\end{align*}
Hence, we can write 
$$\kappa(\phi_{j})=q_j(z)\cdot s(z),$$
where $q_j(z)=2\phi_{j}(W_1,W_2)/z_3$, $j=1,\ 2$, and $s(z)=z_1dz_2\wedge dz_3-z_2dz_1\wedge dz_3+z_3dz_1\wedge dz_2.$

The case of $\kappa$ on $HC^{2}$ is similar. Let $f$ be the linear pencil $A(z)$, where $A=(A_1,\ A_2,\ A_3,\ A_4)$ is a 4-tuple of elements in $\mathcal{A_{\theta}}$.
Clearly, Theorem 3.2 holds for $\psi_1$. For $\psi_2$, one observes that $\psi_2(I, x_1,\ x_2)=\psi_2(I, x_2,\ x_1)$ and 
$\psi_2(x_1,\ x_2,\ x_2)=0$ for every $x_1,\ x_2\in \mathcal{A_{\theta}}$. Using these observations and the cyclicity of $\psi_2$ while going through the proof of Theorem 3.2, we see that Theorem 3.2 also holds for $\psi_j(\omega_A,\omega_A,\omega_A)$.

\end{exam}

\addcontentsline{toc}{section}{\bf{Bibliography}}

\bibliographystyle{plain}	
\bibliography{mybib}
\end{document}